\documentclass{daj}

\usepackage{amsmath,amssymb,amsthm,bm}

\newcommand{\ssquare}{\scalebox{0.6}{$\square$}}
\newcommand{\mb}{\mathbf}
\newcommand{\Bias}{\operatorname{Bias}}
\newcommand{\Hamm}{\operatorname{Hamm}}
\newcommand{\Bohr}{\operatorname{Bohr}}
\newtheorem{theorem}{Theorem}
\newtheorem{lemma}[theorem]{Lemma}
\newtheorem{proposition}[theorem]{Proposition}
\newtheorem{corollary}[theorem]{Corollary}
\numberwithin{theorem}{section}
\numberwithin{equation}{section}

\theoremstyle{definition}

\newtheorem{observation}[theorem]{Observation}
\newtheorem{definition}[theorem]{Definition}
\newtheorem{problem}[theorem]{Problem}
\newtheorem{remark}[theorem]{Remark}
\newtheorem{example}[theorem]{Example}

\dajAUTHORdetails{%
  title = {Separating Bohr Denseness from Measurable Recurrence}, 
  author = {John T. Griesmer},
  plaintextauthor = {John T. Griesmer},
  runningtitle = {Separating Bohr denseness from recurrence},
  keywords = {Bohr topology, measurable recurrence},
}   

\dajEDITORdetails{%
   year={2021},
   number={9},
   received={23 June 2020},   
   published={2 September 2021},  
   doi={10.19086/da.26859},       
}   

\begin{document}

\begin{frontmatter}[classification=text]


\author[jtg]{John T. Griesmer}

\begin{abstract}
		We prove that there is a set of integers $A$ having positive upper Banach density whose difference set $A-A:=\{a-b:a,b\in A\}$ does not contain a Bohr neighborhood of any integer, answering a question asked by Bergelson, Hegyv\'ari, Ruzsa, and the author, in various combinations.   In the language of dynamical systems, this result shows that there is a set of integers $S$ which is dense in the Bohr topology of $\mathbb Z$ and which is not a set of measurable recurrence.	

Our proof yields the following stronger result: if $S\subseteq \mathbb Z$ is dense in the Bohr topology of $\mathbb Z$, then there is a set $S'\subseteq  S$ such that $S'$ is dense in the Bohr topology of $\mathbb Z$ and for all $m\in \mathbb Z,$ the set $(S'-m)\setminus \{0\}$ is not a set of measurable recurrence.
\end{abstract}
\end{frontmatter}


\section{Introduction}

\subsection{Difference sets} As usual $\mathbb Z$ denotes the set of integers, $\mathbb R$ denotes the real numbers with the usual topology, and $\mathbb T$ denotes $\mathbb R/\mathbb Z$ with the quotient topology.  For $A, B\subseteq \mathbb Z$, we let $A+B$ denote the \emph{sumset} $\{a+b: a\in A, b\in B\}$ and $A-A$ the \emph{difference set} $\{a-b: a, b\in A\}$.  If $c\in \mathbb Z$ the \emph{translate} of $A$ by $c$ is $A-c:=\{a-c:a\in A\}$.  The \emph{Bohr topology} of $\mathbb Z$ is the weakest topology on $\mathbb Z$ making all homomorphisms from $\mathbb Z$ into $\mathbb T$ continuous. We call neighborhoods in this topology \emph{Bohr neighborhoods}; see \S\ref{sec:BohrNhoods} for an explicit description.  We say that $S$ is \emph{Bohr dense} if $S$ is dense with respect to the Bohr topology.  We write $d^*(A)$ for the upper Banach density of a set of integers $A$, defined as $d^*(A):=\limsup_{n\to \infty} \sup_{k\in \mathbb Z} \frac{|A\cap \{k+1,\dots,k+n \}|}{n}$.
	
The following problem was posed first in \cite{RuzsaBook} and subsequently in \cite{BergelsonRuzsa}, \cite{GriesmerIsr}, and \cite{HegyvariRuzsa}.

\begin{problem}\label{q:Main}
	Prove or disprove: for all $A\subseteq \mathbb Z$ having $d^*(A)>0$, there is an $n\in \mathbb Z$ such that $A-A$ contains a Bohr neighborhood of $n.$
\end{problem}

Our main result disproves the statement in Problem \ref{q:Main}.

\begin{theorem}\label{thm:Main}
For all $\varepsilon>0$, there are sets $S, A\subseteq \mathbb Z$ such that $S$ is dense in the Bohr topology of $\mathbb Z$, $d^*(A)>\frac{1}{2}-\varepsilon$, and $(A-A)\cap S=\varnothing$.
\end{theorem}
The set $A-A$ in Theorem \ref{thm:Main}  does not contain a Bohr neighborhood, since $S\cap B\neq \varnothing$ for every Bohr neighborhood $B$.

If $\delta\geq 0$, we say  $S$ is \emph{$\delta$-nonrecurrent} if there is a set $A\subseteq \mathbb Z$ having $d^*(A)>\delta$ and $(A-A)\cap S=\varnothing$.  The proof of Theorem \ref{thm:Main} yields the following stronger statement.

\begin{theorem}\label{thm:Strong}  If $S\subseteq \mathbb Z$ is Bohr dense and $\delta<\frac{1}{2}$, then there is a Bohr dense $\delta$-nonrecurrent subset $S'\subseteq S$.
\end{theorem}

Repeatedly applying Theorem \ref{thm:Strong} produces the following corollary, showing that there are Bohr dense sets which are very far from being sets of \emph{measurable recurrence} -- see \S\ref{sec:Systems} for definition of this term.

\begin{corollary}\label{cor:Strong}
If $S\subseteq \mathbb Z$ is Bohr dense then there is a Bohr dense set $S'\subseteq S$ 	such that for all $m\in \mathbb Z$, the set $(S'-m)\setminus \{0\}$ is not a set of measurable recurrence.
\end{corollary}

\begin{remark}\label{rem:DiffSetTop}
  In \cite{RuzsaBook}, Ruzsa defines the \emph{difference set topology} to be the  topology on $\mathbb Z$ generated by translates of sets of the form $A-A$, where $A\subseteq \mathbb Z$ has positive upper Banach density.  A set $S\subseteq \mathbb Z$ is a set of measurable recurrence if and only if $0$ lies in the closure of $S$ with respect to this topology, while $S-m$ is a set of measurable recurrence if and only if $m$ lies in the closure of $S$.  In these terms, Corollary \ref{cor:Strong} states that every Bohr dense set $S\subseteq \mathbb Z$ contains a Bohr dense subset $S'$ which is closed, discrete, and nowhere dense in the difference set topology.
\end{remark}

\section{Measure preserving systems; outline of proof}\label{sec:Systems}

\subsection{Measure preserving systems}  By \emph{measure preserving system} we mean a triple $(X,\mu,T)$ where $(X,\mu)$ is a probability measure space and $T:X\to X$ is an invertible transformation preserving $\mu$: for every measurable set $D\subseteq X$, $T^{-1}D$ is measurable and $\mu(T^{-1}D)=\mu(D)$.

We say that $S\subseteq \mathbb Z$ is a \emph{set of measurable recurrence} if for every measure preserving system $(X,\mu,T)$ and every measurable set $D\subseteq X$ with $\mu(D)>0$ there is an $n\in S$ such that $D\cap T^{n}D\neq \varnothing$.

Correspondence principles such as \cite[Proposition 3.1]{BHK} or \cite[Theorem 3.18]{Fbook} allow us to phrase the concept of $\delta$-nonrecurrence in terms of measure preserving systems.  Here is the correspondence principle we need for our proofs.

\begin{lemma}\label{lem:nonrecurrenceForms}
  Let $\delta>0$ and $S\subseteq \mathbb Z$.  The following are equivalent:

  \begin{enumerate}
    \item[(i)] There is a measure preserving system $(X,\mu,T)$ and $D\subseteq X$ with $\mu(D)>\delta$ such that $\mu(D\cap T^sD)=\varnothing$ for all $s\in S$.

    \item[(ii)] There exists $A\subseteq \mathbb Z$ with $d^*(A)>\delta$ such that $(A-A)\cap S=\varnothing$.

    \item[(iii)] There is a $\delta'>\delta$ such that for all $n\in \mathbb N$ there exists $A_n\subseteq \{0,\dots,n-1\}$  with $|A_n|\geq \delta'n$ and $(A_n-A_n)\cap S=\varnothing$.
  \end{enumerate}
\end{lemma}

The following lemma is crucial in constructions of $\delta$-nonrecurrent sets; it follows from Theorems 1 and  2 of \cite{RuzsaIntersectivity84}. It is  also a consequence of the proof of  \cite[Theorem 2.1]{ForrestThesis}.
\begin{lemma}\label{lem:DensityCompactness}  Let $\delta>0$.
Let $S\subseteq \mathbb Z$ and $0\leq \delta <\delta'$. If every finite subset of $S$ is $\delta'$-nonrecurrent, then $S$ is $\delta$-nonrecurrent.
\end{lemma}
Lemmas \ref{lem:nonrecurrenceForms} and \ref{lem:DensityCompactness} are well known but rarely collected together and stated as we have here, so we prove them in \S\ref{sec:Appendix}.

\subsection{Torus rotations and Rohlin towers} Fixing $d\in \mathbb N$ and $\boldsymbol\alpha\in \mathbb T^d$, the corresponding \emph{torus rotation} is the measure preserving system $(\mathbb T^d,\mu,R)$, where $R\mb x = \mb x + \bm\alpha$ and $\mu$ is Haar probability measure on $\mathbb T^d$.  We say that $(\mathbb T^d,\mu,R)$ is \emph{minimal} if $\{n\bm\alpha:n\in \mathbb Z\}$ is dense in $\mathbb T^d$.

A \emph{Rohlin tower} for a measure preserving system $(X,\mu,T)$ is a collection of mutually disjoint measurable subsets of $X$ having the form $\mathcal T=\{E, TE, T^2E, \dots, T^{N-1}E\}$.  We say the tower has \emph{base} $E$, \emph{height} $N$, and we call the elements of $\mathcal T$ the \emph{levels} of $\mathcal T$.  A set $D\subseteq X$ is \emph{$\mathcal T$-measurable} if $D$ is a union of levels of $\mathcal T$.

From now on we write $[N]$ for the interval $\{0,\dots, N-1\}$ in $\mathbb Z$.  If $S\subseteq \mathbb Z$ is a finite $\delta$-nonrecurrent set and $\mathcal T$ is a Rohlin tower of height $N$ and base $E$, we say that $\mathcal T$ \emph{witnesses the $\delta$-nonrecurrence of $S$} if there is a set $A\subseteq [N]$ such that $A+S\subseteq [N]$, $A\cap (A+S)=\varnothing$, and $|A|\mu(E)>\delta$.  Note that this implies $D:=\bigcup_{n\in A} T^nE$ satisfies $\mu(D)>\delta$ and $D\cap T^sD=\varnothing$ for all $s\in S$.

\subsection{Extending \texorpdfstring{$\delta$}{d}-nonrecurrent sets with pairs of Rohlin towers} Proposition \ref{prop:RohlinTd} provides a special class of Rohlin towers which are the focal point of our main argument.  Lemma \ref{lem:RohlinExtend} indicates that such towers can be used to construct $\delta$-nonrecurrent sets with prescribed properties.  

\begin{lemma}\label{lem:RohlinExtend}
Let $S\subseteq \mathbb Z$ be finite, $\delta>0$, and let $(X,\mu,T)$ be a measure preserving system. Let
\[\mathcal T = \{E, TE,\dots, T^{N-1}E\}, \qquad \mathcal T'=\{E', TE',\dots, T^{N-1}E'\}
\] be Rohlin towers for $T$ with $E\subseteq E'$ and define $S':=\{n\in \mathbb Z: T^nE\subseteq E'\}$.   If $\mathcal T$ witnesses the $\delta$-nonrecurrence of $S$, then $S\cup (S+S')$ is $\delta$-nonrecurrent.
\end{lemma}

\begin{remark}
  The hypothesis  $E\subseteq E'$ in Lemma \ref{lem:RohlinExtend} implies $0\in S'$, and thus $S\subseteq S+S'$, so we could simply write ``$S+S'$ is $\delta$-nonrecurrent'' in the conclusion.  Instead, we want to emphasize that the new $\delta$-nonrecurrent set contains $S$.
\end{remark}

\begin{proof}
Assuming $\mathcal T$, $\mathcal T'$, and $\delta$ are as in the hypothesis, there is an $A\subseteq [N]$ such that $A+S\subseteq [N]$, $A\cap (A+S)=\varnothing$, and $|A|\mu(E)>\delta$.  Then
\[D:=\bigcup_{n\in A} T^nE, \qquad D':=\bigcup_{n\in A} T^nE'\]
 each have measure strictly greater than $\delta$, and the disjointness of the levels of $\mathcal T$ implies
\begin{equation}\label{eqn:DsD}
  D\cap T^sD=\varnothing \quad \text{and} \quad D'\cap T^sD'=\varnothing \quad \text{for all }  s\in S.
\end{equation}
To prove that $S\cup (S+S')$ is $\delta$-nonrecurrent it therefore suffices to prove
\begin{equation}\label{eqn:DssD}
D\cap T^{s+s'}D=\varnothing \quad \text{for all} \quad s\in S, s'\in S',
\end{equation}
so fix $s\in S$ and $s'\in S'$. Note that $D\subseteq D'$ since $E\subseteq E'$, and that $T^{s'}E\subseteq E'$, by the definition of $S'$.  Then $T^{s'}D=\bigcup_{a\in A} T^{a+s'} E\subseteq \bigcup_{a\in A} T^a E' = D'$.  Now $T^{s+s'}D\subseteq T^{s}D'$, so the containment $D\subseteq D'$ and the disjointness of $D'$ from $T^{s}D'$ implies $D\cap T^{s+s'}D=\varnothing$.  We have proved the lemma, as (\ref{eqn:DsD}) and (\ref{eqn:DssD}) imply that $S\cup (S+S')$ is $\delta$-nonrecurrent. \end{proof}

\subsection{Outline of the main argument}\label{sec:Outline} Lemma \ref{lem:RohlinExtend} forms the basis of an inductive construction of a $\delta$-nonrecurrent set which is Bohr dense.  This construction requires two compactness properties: first, that Bohr denseness can be approximated by \emph{$k$-Bohr denseness} (Definition \ref{def:BohrDenseness}), which in turn can be approximated using finite subsets of $\mathbb Z$ (Lemma \ref{lem:BohrIntersectiveCompactness}).  The corresponding compactness property for measurable recurrence is provided by Lemma \ref{lem:DensityCompactness}.  Starting with a finite $\delta$-nonrecurrent set $S_1$, we use Lemma \ref{lem:WitnessesExist} to find an $N\in \mathbb N$ and a finite set $A\subseteq [N]$ witnessing the $\delta$-nonrecurrence of $S_1$, meaning $|A|>\delta N$, $A\cap (A+
S_1)=\varnothing$, and $A+S_1\subseteq [N]$.  We then use Proposition \ref{prop:RohlinTd} to find a minimal torus rotation $(\mathbb T^d,\mu,R)$ and Rohlin towers $\mathcal T$, $\mathcal T'$ as in Lemma \ref{lem:RohlinExtend} with $|A|\mu(E)>\delta$ such that the set $\{n:R^nE\subseteq E'\}$ contains a \emph{Bohr-Hamming ball} $BH$ (Definition \ref{def:BH}), which itself is $k$-Bohr dense (Lemma \ref{lem:BHiskBohrDense}).  Lemma \ref{lem:RohlinExtend} then implies $S_1\cup (S_1+BH)$ is $\delta$-nonrecurrent.  Since $k$-Bohr denseness is translation invariant, we will get that $S_1\cup (S_1+BH)$ is $k$-Bohr dense, and Lemma \ref{lem:BohrIntersectiveCompactness} will allow us to chose a finite subset $S_2\subseteq S_1\cup (S_1+BH)$ which is approximately $k$-Bohr dense. Since $S_1$ is finite we may include $S_1$ in $S_2$.  Repeating this argument, we produce a sequence of sets $S_1\subseteq S_2\subseteq S_3\subseteq \cdots$ where each $S_k$ is approximately $k$-Bohr dense and $\delta$-nonrecurrent.  The union $\bigcup_{k\in \mathbb N} S_k$ will be the desired Bohr dense $\delta$-nonrecurrent set.

\subsection{Organization of the article}  The argument outlined in \S\ref{sec:Outline} is the main one used in the proof of Theorem \ref{thm:Main}; complete details are provided in \S\ref{sec:BHsubsec}. A superficial modification of this argument will prove Theorem \ref{thm:Strong} as well.  As Theorem \ref{thm:Main} is a special case of Theorem \ref{thm:Strong}, we address only the latter in the sequel.  Corollary \ref{cor:Strong} follows from a straightforward diagonalization based on repeated application of Theorem \ref{thm:Strong}.

In \S\ref{sec:BohrNhoods} we state definitions related to Bohr neighborhoods and prove some standard compactness properties regarding the Bohr topology. Section \ref{sec:BHsubsec} introduces \emph{Bohr-Hamming balls}, whose relevant properties  are recorded in Lemma \ref{lem:BHiskBohrDense} and Proposition \ref{prop:RohlinTd}; these are proved in \S\ref{sec:Independence} and \S\S\ref{sec:InFp}-\ref{sec:RohlinTowers}, respectively. Lemma \ref{lem:2Pieces} combines Lemma \ref{lem:BHiskBohrDense} and Proposition \ref{prop:RohlinTd} to form the inductive step in the proof of Theorem \ref{thm:Strong}.  The proofs of Theorem \ref{thm:Strong}  and Corollary \ref{cor:Strong} are presented immediately after the proof of Lemma \ref{lem:2Pieces}.

\section{Bohr neighborhoods}\label{sec:BohrNhoods}  We identify $\mathbb T$ with the interval $[0,1)\subseteq \mathbb R$ when defining elements and subsets of $\mathbb T$.
For $x\in \mathbb T$, let $\tilde{x}$ denote the unique element in $[0,1)$ such that $x = \tilde{x}+\mathbb Z,$ and define $\|x\|:=\min\{|\tilde{x}-n|: n\in \mathbb Z\}$.  For $d\in \mathbb N$ and $\mb x = (x_1,\dots, x_d)\in \mathbb T^d$, let $\|\mb x\|:=\max_{j\leq d} \|x_j\|$.

Fixing $d\in \mathbb N$, $\bm{\alpha}\in \mathbb T^d$, and  a nonempty open set $U\subseteq \mathbb T^d$, the \emph{Bohr neighborhood} determined by these parameters is
\[
B(\bm{\alpha}; U):= \{n\in \mathbb Z: n\bm\alpha\in U\}.
\]
We say that $B(\bm{\alpha};U)$ has  \emph{rank}  $d$.  Observe that $B(\bm{\alpha};U)$ may be empty, as we make no assumptions on $\bm{\alpha}$.  However, when $0_{\mathbb T^d}\in U$, $B(\bm{\alpha};U)$ is nonempty, as it contains $0$.  The \emph{Bohr topology} on $\mathbb Z$ is the weakest topology containing $B(\bm\alpha;U)$ for every $\bm\alpha\in \mathbb T^d$ and open $U\subseteq  \mathbb T^d$, for every $d\in\mathbb N$.

Given $\bm\alpha\in \mathbb T^d$ and $\varepsilon>0$, we define \[
\Bohr_0(\bm\alpha,\varepsilon):=\{n\in \mathbb Z: \|n\bm\alpha\|<\varepsilon\}
\]
to be a \emph{basic Bohr neighborhood of $0$} having \emph{rank}  $d$ and \emph{radius} $\varepsilon$.  These form a neighborhood base around $0$ for the Bohr topology.  For a given $n\in \mathbb Z$, the collection of translates
\[
\{B+n:B \text{ is a basic Bohr neighborhood of } 0\}
\]
forms a neighborhood base at $n$ in the Bohr topology.

\begin{example}\label{ex:Odd}
  The set of odd integers $B:=2\mathbb Z+1$ is the Bohr neighborhood $B(\alpha;U)$, determined by $\alpha=\frac{1}{2}\in\mathbb T$ and $U=\mathbb T\setminus \{0_{\mathbb T}\}$.   For every $\delta<\frac{1}{2}$, $B$ is $\delta$-nonrecurrent, since the set $2\mathbb Z$ of even integers has upper Banach density $\frac{1}{2}$, while $(2\mathbb Z-2\mathbb Z)\cap B=\varnothing$.
\end{example}

\begin{observation}\label{obs:Basic}
If $m\in B(\bm{\alpha};U)$, then the translate $B(\bm{\alpha};U)-m$ is a Bohr neighborhood of $0$, and therefore contains a basic Bohr neighborhood of $0$. Consequently, every nonempty Bohr neighborhood having rank at most $d$ contains a translate of a basic Bohr neighborhood of $0$ having rank at most $d$.
\end{observation}

\begin{definition}[Bohr denseness and its approximations]\label{def:BohrDenseness} We say that $S\subseteq \mathbb Z$ is

\begin{enumerate}
\item[$\cdot$] \emph{Bohr recurrent} if $S\cap B\neq \varnothing$ for every Bohr neighborhood of $0$.

\item[$\cdot$] \emph{$d$-Bohr recurrent} if $S\cap B\neq \varnothing$ for every Bohr neighborhood of $0$ having rank at most $d$.

\item[$\cdot$] \emph{$(d,\varepsilon)$}\emph{-Bohr recurrent} if $S\cap B\neq \varnothing$ for every basic Bohr neighborhood of $0$ with rank at most $d$ and radius at least $\varepsilon$.

\item[$\cdot$] \emph{Bohr dense} if $S\cap B\neq \varnothing$ for every nonempty Bohr neighborhood $B$.

\item[$\cdot$] \emph{$d$-Bohr dense} if $S\cap B\neq \varnothing$ for every nonempty Bohr neighborhood with rank at most $d$.  Equivalently, $S$ is $d$-Bohr dense if $S-m$ is $d$-Bohr recurrent for all $m\in \mathbb Z$.
\end{enumerate}
\end{definition}
The equivalence asserted in the last item above is due to Observation \ref{obs:Basic} and the fact that $(S-m)\cap B$ is a translate of $S\cap (B+m)$.  The next observation follows immediately from the relevant definitions and Observation \ref{obs:Basic}.

\begin{observation}\label{obs:Easy} Let $S\subseteq \mathbb Z$.  Then
\begin{enumerate}
	\item[(i)]
	$S$ is $(d,\varepsilon)$-Bohr recurrent if and only if for all $\bm\alpha\in\mathbb T^d$, there exists $s\in S$ such that $\|s\bm\alpha\|<\varepsilon$.
	
	\item[(ii)] 	$S$ is Bohr dense if and only if for all $d\in \mathbb N$, $\varepsilon>0$, and $m\in \mathbb Z,$ the set $S-m$ is $(d,\varepsilon)$-Bohr recurrent.
\end{enumerate}

\end{observation}
The next lemma is an instance of compactness required for our proofs.

\begin{lemma}\label{lem:BohrIntersectiveCompactness}  Let $d\in \mathbb N$.
	\begin{enumerate}
		\item[(i)] If $S\subseteq \mathbb Z$ is $d$-Bohr recurrent, then for all $\varepsilon>0$, there is a finite set $S'\subseteq S$ such that $S'$ is $(d,\varepsilon)$-Bohr recurrent.
		
		\item[(ii)] If $S\subseteq \mathbb Z$ is $d$-Bohr dense, then for all $M\in \mathbb N$ and all $\varepsilon>0$, there exists a finite set $S'\subseteq S$ such that for all $m\in \mathbb Z$ with $|m|\leq M$, the translate $S'-m$ is $(d,\varepsilon)$-Bohr recurrent.
	\end{enumerate}

\end{lemma}

\begin{proof}
	We prove Part (i) by proving its contrapositive: assuming $\varepsilon>0$ and that for every finite $S'\subseteq S$ there is an $\bm{\alpha}\in \mathbb T^d$ with $\|s\bm{\alpha}\|\geq \varepsilon$ for all $s\in S'$, we will find an $\bm{\alpha}\in \mathbb T^d$ such that $\|s\bm{\alpha}\|\geq \varepsilon$ for all $s\in S$.  Enumerate $S$ as $(s_j)_{j\in \mathbb N}$, and for each $n$ choose $\bm{\alpha}_n\in \mathbb T^d$ such that $\|s_j\bm{\alpha}_n\|>\varepsilon$ for all $j\leq n$; this is possible due to our hypothesis on finite subsets of $S$.  Choose a convergent subsequence $(\bm{\alpha}_{n_k})_{k\in \mathbb N}$ and call the limit $\bm{\alpha}$.  Now for all $s\in S$, we have $\|s\bm{\alpha}_{n_k}\|\to \|s\bm{\alpha}\|$, and our choice of $\bm{\alpha}_n$ means that $\|s\bm{\alpha}_{n_k}\|>\varepsilon$ for all but finitely many $k$.  Thus $\|s\bm{\alpha}\|\geq \varepsilon$ for all $s\in S$.
	
	Part (ii) follows from Part (i), Observation \ref{obs:Easy}, and the definition of ``$d$-Bohr dense''.
\end{proof}

The next lemma is essentially Lemma 5.11 of \cite{BadeaGrivauxMatheronRigidityKazhdan}.  We use it to derive Corollary \ref{cor:Strong} from Theorem \ref{thm:Strong}.

\begin{lemma}\label{lem:BohrSelective}
	Let $(S_n)_{n\in \mathbb N}$ be a sequence of Bohr dense subsets of $\mathbb Z$.  Then there is a sequence of finite sets $R_n\subseteq S_n$ such that $\bigcup_{n\in \mathbb N} R_n$ is Bohr dense.
\end{lemma}

\begin{proof}
	The Bohr denseness of $S_n$ and Lemma \ref{lem:BohrIntersectiveCompactness} allow us to choose, for each $n$, a finite subset $R_n\subseteq  S_n$ such that $R_n-m$ is $(n,1/n)$-Bohr recurrent for each $m$ with $|m|<n$.  Observation \ref{obs:Easy} then implies that $\bigcup_{n\in \mathbb N} R_n$ is Bohr dense.
\end{proof}

\section{Bohr-Hamming balls; proof of Theorem \ref{thm:Strong} and Corollary \ref{cor:Strong}}\label{sec:BHsubsec}

\subsection{Bohr-Hamming Balls} For $\varepsilon>0$, $d\in \mathbb N$, and $\mb x=(x_1,\dots,x_d)\in \mathbb T^d$, let
\[
w_\varepsilon(\mb x):= |\{j: \|x_j\|\geq \varepsilon\}|.
\]
So $w_\varepsilon(\mb x)$ is the number of coordinates of $\mb x$ differing from $0$ by at least $\varepsilon$.  Following \cite{Katznelson}, we call $\bm\alpha\in \mathbb T^d$ a \emph{generator} if $\{n\bm\alpha: n\in\mathbb Z\}$ is dense in $\mathbb T^d$.
\begin{definition}\label{def:BH}
  Let $k< d\in \mathbb N$, $\varepsilon>0$, and $\bm\alpha\in \mathbb T^d$.  The \emph{Bohr-Hamming ball with rank $d$ and radius $(k,\varepsilon)$ around $0$} determined by $\bm\alpha$ is
  \[
  BH(\bm\alpha;k,\varepsilon):=\{n\in \mathbb Z: w_\varepsilon(n\bm\alpha)\leq k\}.
  \]
So $n\in BH(\bm\alpha;k,\varepsilon)$ if at most $k$ coordinates of $n\bm\alpha$ differ from $0$ by at least $\varepsilon$.  If $\bm\alpha$ is a generator, we say that $BH(\bm\alpha;k,\varepsilon)$ is \emph{proper}.
\end{definition}

The next lemma, implicit in Section 2 of \cite{Katznelson},  asserts a useful relation between Bohr-Hamming balls and Bohr neighborhoods.
\begin{lemma}\label{lem:BHiskBohrDense} Let $k<d\in \mathbb N$ and $\varepsilon>0$.  If $BH$ is a proper Bohr-Hamming ball with rank $d$ and radius $(k,\varepsilon)$ and $B$ is a nonempty Bohr neighborhood with rank $k$, then $BH\cap B$ contains a nonempty Bohr neighborhood with rank $d$.

   Consequently, if $S\subseteq \mathbb Z$ is $d$-Bohr dense, then $S\cap BH$ is $k$-Bohr dense.
\end{lemma}
We prove Lemma \ref{lem:BHiskBohrDense} in \S\ref{sec:Independence}, following closely the proof of Lemmas 2.2 and 2.3 of \cite{Katznelson}.

The proof of Theorem \ref{thm:Strong} is mostly contained in the following lemma, which is an easy consequence of the subsequent proposition.

\begin{lemma}\label{lem:2Pieces}
Let $\delta>0$ and $k\in \mathbb N$.  If $S\subseteq \mathbb Z$ is finite and $\delta$-nonrecurrent, then there is an $\eta>0$ and a proper Bohr-Hamming ball $BH$ of radius $(k,\eta)$ such that $S+BH$ is $\delta$-nonrecurrent.
\end{lemma}

\begin{proposition}\label{prop:RohlinTd}
For every $k\in \mathbb N$, $\varepsilon>0$, and prime $p$, there exist $d\in \mathbb N$, a minimal torus rotation $(\mathbb T^d,\mu,R)$ by $\bm\alpha\in \mathbb T^d$, $\eta>0$, a proper Bohr-Hamming ball $BH=BH(\bm\alpha;k,\eta)$ with rank $d$, and Rohlin towers
\begin{equation*}
\mathcal T = \{R^nE : 0\leq n \leq p-1\}, \qquad \mathcal T' = \{R^nE' : 0\leq n \leq p-1\}
\end{equation*}
such that $\mu(E)>\frac{1-\varepsilon}{p}$, $E\subseteq E'$, and $R^nE\subseteq E'$ for all $n\in BH$.
\end{proposition}
The proof of Proposition \ref{prop:RohlinTd} occupies \S\S\ref{sec:InFp}-\ref{sec:RohlinTowers}.  

 We need one more standard lemma for the proof of Lemma \ref{lem:2Pieces}.
\begin{lemma}\label{lem:WitnessesExist}
If $S\subseteq \mathbb Z$ is a finite $\delta$-nonrecurrent set, then for all sufficiently large $N$, there exists $A\subseteq [N]$ with $|A|>\delta N$, $A+S\subseteq [N]$, and $A\cap (A+S)=\varnothing$.
\end{lemma}
The condition $A\cap (A+S)=\varnothing$ is equivalent to $(A-A)\cap S=\varnothing$; we will use this from time to time without comment.

\begin{proof}
  Assume $S$ is finite and $\delta$-nonrecurrent and  let $M=\max\{|s|+1:s\in S\}$. We choose, by Part (iii) of Lemma \ref{lem:nonrecurrenceForms}, a $\delta'>\delta$ such that for all $N\in \mathbb N$, there is an $A_N\subseteq [N]$ with $|A_N|\geq \delta'N$ and $A_N\cap (A_N+S)=\varnothing$.  Choose $N$ large enough that $\delta'N-2M> \delta N$.  Letting $A = A_N\cap [M-1,N-M-1]$, we have $|A|\geq |A_N|-2M>\delta'N-2M>\delta N$, $A+S\subseteq [N]$, and $A\cap (A+S)=\varnothing$.
\end{proof}

\begin{proof}[Proof of Lemma \ref{lem:2Pieces}] To prove Lemma \ref{lem:2Pieces}, we will apply Lemma \ref{lem:RohlinExtend} to the Rohlin towers provided by Proposition \ref{prop:RohlinTd}. Let $k\in \mathbb N$ and assume $S\subseteq \mathbb Z$ is finite and $\delta$-nonrecurrent. We will find a minimal torus rotation $(\mathbb T^d,\mu,R)$ by an $\bm\alpha\in \mathbb T^d$, a measurable set $D\subseteq \mathbb T^d$ having $\mu(D)>\delta$, and a Bohr-Hamming ball $BH=BH(\bm\alpha;k,\eta)$ such that
\begin{equation}\label{eqn:DRnD}
	D\cap R^{n}D = \varnothing \quad \text{for all } n\in S+BH.
\end{equation}
To construct $D$, we first apply Lemma \ref{lem:WitnessesExist} to find a  prime $p$ and $A\subseteq [p]$ having $|A|>\delta p$ such that $A\cap (A+S)=\varnothing$ and $A+S\subseteq [p]$; this is possible due to our assumptions on $S$.  Fix $\varepsilon>0$ so that $|A|\frac{(1-\varepsilon)}{p}>\delta$ and invoke Proposition \ref{prop:RohlinTd} with this $\varepsilon$.  We form $D$  by copying $A$ into levels of the tower $\mathcal T$ provided by Proposition \ref{prop:RohlinTd}:
\[D:=\bigcup_{a\in A} R^a E.\] By our choice of $\varepsilon$ and the mutual disjointness of the levels of $\mathcal T$, we have
\[
\mu(D) = |A|\mu(E) > \frac{|A|(1-\varepsilon)}{p} > \delta.
\] To prove that $D$ satisfies (\ref{eqn:DRnD}), observe that the hypotheses of Lemma \ref{lem:RohlinExtend} hold with $p$ in place of $N$.  Proposition \ref{prop:RohlinTd} states that $BH\subseteq \{n:R^nE\subseteq E'\}$,  so we may   cite Lemma \ref{lem:RohlinExtend} with $BH$ in place of $S'$ and conclude that (\ref{eqn:DRnD}) holds.  \end{proof}

\subsection{Proof of  Theorem \ref{thm:Strong} and Corollary \ref{cor:Strong}}

Recall the statement of Theorem \ref{thm:Strong}: if $S\subseteq \mathbb Z$ is Bohr dense and $\delta<\frac{1}{2}$, then there is Bohr dense $\delta$-nonrecurrent subset $S'\subseteq S$.

\begin{proof}[Proof of Theorem \ref{thm:Strong}]
  Let $S\subseteq \mathbb Z$ be  Bohr dense and let $\delta<\frac{1}{2}$.    By Lemma \ref{lem:DensityCompactness} it suffices to find $\delta'>\delta$ and a Bohr dense set $S'\subseteq S$ such that every finite subset $S''\subseteq S'$ is $\delta'$-nonrecurrent.  Fixing $\delta'$ with $\delta < \delta' <\frac{1}{2}$, we will construct an increasing sequence $S_1 \subseteq S_2 \subseteq \dots$ of subsets of $S$ such that each $S_k$ is $\delta'$-nonrecurrent and satisfies the following condition:
  \begin{equation}\label{eqn:SkTranslates}
  \text{for all } m\in \mathbb Z \text{ with } |m|\leq k, \text{ the translate } S_k-m \text{ is } (k,1/k)\text{-Bohr recurrent}.
  \end{equation}
  To construct $S_1$, we find an odd integer $s_1\in S$, and let $S_1=\{s_1\}$.  Such $s_1$ exists, as the odd integers form a Bohr neighborhood (Example \ref{ex:Odd}) and $S$ is Bohr dense.  Now $S_1$ is $\delta'$-nonrecurrent, as the set of odd numbers is $\delta'$-nonrecurrent for every $\delta'<\frac{1}{2}$.

For the inductive step of the construction, we assume $S_{k-1}$ is a finite $\delta'$-nonrecurrent subset of $S$. We apply Lemma \ref{lem:2Pieces} to find a proper Bohr-Hamming ball $BH$ with radius $(k,\eta)$  such that $S_{k-1}+BH$ is $\delta'$-nonrecurrent.  Lemma \ref{lem:BHiskBohrDense} implies $S\cap (S_{k-1}+BH)$ is $k$-Bohr dense, and Lemma \ref{lem:BohrIntersectiveCompactness} provides a finite subset $S_k$ of $S\cap (S_{k-1}+BH)$ satisfying (\ref{eqn:SkTranslates}). Since $0\in BH$ we have $S_{k-1}\subseteq S_{k-1}+BH$. The finiteness of $S_{k-1}$ and the latter containment means we can choose $S_k$ to satisfy $S_{k-1}\subseteq  S_k$ as well.

  Letting $S':=\bigcup_{k\in \mathbb N} S_k$, we have that every finite subset of $S'$ is contained in one of the sets $S_k$, and each $S_k$ is $\delta'$-nonrecurrent, so Lemma \ref{lem:DensityCompactness} implies $S'$ is $\delta$-nonrecurrent.  The Bohr denseness of $S'$ follows from (\ref{eqn:SkTranslates}) and Observation \ref{obs:Easy}.
  \end{proof}

The next lemma records two elementary facts for the proof of Corollary \ref{cor:Strong}.  \begin{lemma}\label{lem:Elementary} Let $R, S\subseteq \mathbb Z$.
\begin{enumerate}
\item[(i)] If neither $R$ nor $S$ is a set of measurable recurrence then $R\cup S$ is not a set of measurable recurrence.

\item[(ii)] If $S\subseteq \mathbb Z$ is finite then $S\setminus \{0\}$ is not a set of measurable recurrence.
\end{enumerate}
\end{lemma}
Part (i) is proved by taking the cartesian product of measure preserving systems witnessing the nonrecurrence of $R$ and $S$. Part (ii) follows from considering a group rotation on $\mathbb Z/N\mathbb Z$, where $N=1+\max\{|s|:s\in S\}$.

We now prove Corollary \ref{cor:Strong}, which says that if $S\subseteq \mathbb Z$ is Bohr dense then there is a Bohr dense set $S'\subseteq S$ 	such that for all $m\in \mathbb Z$, the set $(S'-m)\setminus \{0\}$ is not a set of measurable recurrence.

\begin{proof}[Proof of Corollary \ref{cor:Strong}]
	Let $S\subseteq \mathbb Z$ be Bohr dense.  We begin by constructing a decreasing sequence  $S_0\supseteq S_1\supseteq S_2\supseteq \cdots $ of Bohr dense subsets of $S$  such that for each $n$,
\begin{equation}\label{eqn:Neither}
\text{neither } S_n-n \text{ nor } S_n+n \text{ is a set of measurable recurrence.}
\end{equation}  We begin with $n=0$ and apply Theorem \ref{thm:Strong} to find a Bohr dense subset $S_0\subseteq S$ which is not a set of measurable recurrence.  Supposing $S_{n-1}$ is defined and is Bohr dense, then each of its translates is Bohr dense as well, and we may apply Theorem \ref{thm:Strong} to $S_{n-1}-n$ to find a Bohr dense subset $S_{n,0}\subseteq S_{n-1}$ such that  $S_{n,0}-n$ is not a set of measurable recurrence. Repeating this process with $S_{n,0}+n$ in place of $S_{n-1}-n$ produces a Bohr dense set $S_n\subseteq S_{n-1}$ satisfying (\ref{eqn:Neither}). Having constructed $S_n$, Lemma \ref{lem:BohrSelective} provides finite sets $R_n\subseteq S_n$ such that $S':=\bigcup_{n\in \mathbb N} R_n$ is Bohr dense.

To complete the proof we fix $m\in \mathbb Z$ and will show that $m\in \mathbb Z$, $(S'-m)\setminus \{0\}$ is not a set of measurable recurrence.  Observe that $(S'-m)\setminus (S_m-m)$ is finite, as all but finitely many of the $R_n$ are contained in $S_m$.  Thus $S'-m$ can be written as $E\cup (S_m-m)$, where $E$ is finite. Since $S_m-m$ is not a set of measurable recurrence,  Lemma \ref{lem:Elementary} implies that $(S'-m)\setminus\{0\}$ is also not a set of measurable recurrence.
\end{proof}

\section{Rohlin towers in \texorpdfstring{$(\mathbb Z/p\mathbb Z)^d$}{Z/pZd}}\label{sec:InFp}

In \S\ref{sec:RohlinTowers} we prove Proposition \ref{prop:RohlinTd} by constructing certain Rohlin towers for minimal torus rotations.  In this section we prove Lemma \ref{lem:GpTile}, establishing much of the structure of the towers while working in $(\mathbb Z/p\mathbb Z)^d$, where $p$ is a fixed prime.  Section \ref{sec:Copying} explains the routine process of copying this structure into $\mathbb T^d$.

\subsection{Hamming balls in \texorpdfstring{$\mathbb Z/N\mathbb Z$}{Z/NZ}}\label{sec:HammField}

For $N, d\in \mathbb N$ we let $G_N^d$ denote the group $(\mathbb Z/N\mathbb Z)^d$.  We write elements of $G_N^d$ as $\mb x = (x_1,\dots, x_N)$, where $x_j\in \mathbb Z/N\mathbb Z$.  In general we write $\mathbf 0 := (0,\dots,0)$ and $\mathbf 1 := (1,\dots, 1)\in G_N^d$.  If $n\in \mathbb Z$ we write $n\mathbf 1$ for $(n , \dots, n)$.  For $\mb x\in G_N^d$, define
\[w(\mb x):=|\{j: x_j\neq 0 \}|,\] so that $w(\mb x)$ is the number of coordinates of $\mb x$ which are not equal to $0$.  Given $k\in \mathbb N$, let
\[
H_k:=\{\mb x\in G_N^d : w(\mb x) \leq k\}.
\]
So $H_k$ is the set elements of $G_N^d$ which are nonzero in at most $k$ coordinates, otherwise known as the \emph{Hamming ball of radius $k$ around $\mb 0$}.

\begin{lemma}\label{lem:GpTile}
	Let $p\in \mathbb N$ be prime.  For all $k\in \mathbb N$ and all $\varepsilon>0$, there exists $d\in \mathbb N$ and sets $A,$ $A_1\subseteq G_p^d$ such that $|A|>\frac{1-\varepsilon}{p}|G_p^d|$, $A\subseteq A_1$, and $A+H_k\subseteq A_1$,  while the translates
	\[A_1,\,A_1 + \bm 1, \dots,\,A_1+ (p-1)\bm 1,
\]
 are mutually disjoint.
\end{lemma}


The proof of Lemma \ref{lem:GpTile} occupies the remainder of this section.  To construct $A$ and $A_1$ we need sets which are very nearly invariant under translation by elements of $H_k$, and whose translates by $\mb 1$, $\dots$, $(p-1)\mb 1$ are mutually disjoint.  Such sets are defined in \S\ref{sec:BiasCells}, and assembled to form $A$ and $A_1$ in \S\ref{sec:Assembly}.

\subsection{Bias cells}\label{sec:BiasCells}  Fix a prime $p$ for the remainder of this section.  For $t\in \mathbb Z/p\mathbb Z$ and $\mb y = (y_1,\dots, y_d)\in G_p^d$, let
 \[
w(\mb y;t):=|\{j:y_j=t\}|,
\]
so that $w(\mb y; t)$ is the number of coordinates of $\mb y$ which are equal to $t$. We let $\mathcal P$ denote the collection of nonempty proper subsets of $\mathbb Z/p\mathbb Z$.  For each $C\in \mathcal P$ and $k, d\in \mathbb N$, let
\[
\Bias(C,k,d):=\{\mb y\in G_p^d: w(\mb y;t)> \tfrac{d}{p}+k  \text{ if } t\in C,\, w(\mb y;t) < \tfrac{d}{p}-k \text{ if } t\notin C\}.
\]
For example, with $p=3$ and $C = \{0,1\}$, $\Bias(C,5,3000)$ is the set of $\mb y \in G_3^{3000}$ such that more than $1005$ coordinates of $\mb y$ are equal to $0$, more than $1005$ coordinates of $\mb y$ are equal to $1$, and fewer than $995$ coordinates of $\mb y$ are equal to $2$.

The following lemma records some elementary properties of $\Bias(C,k,d)$.
\begin{lemma}\label{lem:Bias}
Let $C, C'\in \mathcal P.$  For all $d, k\in \mathbb N$
\begin{enumerate}

\item[(i)]  $\Bias(C,k,d)+\mb 1= \Bias(C+1,k,d)$,

\item[(ii)] if $C\neq C'$ then $\Bias(C,k,d)\cap \Bias(C',k,d)=\varnothing$.
\end{enumerate}

If $l<k$ then

\begin{enumerate}
\item[(iii)]    $\Bias(C,k,d)\subseteq \Bias(C,l,d)$,

\item[(iv)]   $\Bias(C,k,d) + H_{l}\subseteq \Bias(C, k-l,d)$.
\end{enumerate}
\end{lemma}

\begin{proof}
To prove Part (i), observe that $w(\mb x+\mb 1; t) = w(\mb x;t-1)$ for all $\mb x\in G_p^d$ and all $t\in \mathbb Z/p\mathbb Z$.  If $\mb x$ satisfies the inequalities defining $\Bias(C,k,d)$, it follows that $\mb x+\mb 1$ satisfies the  inequalities defining $\Bias(C+1,k,d)$.

To prove Part (ii) note that if $\mb y$ lies in the intersection written in (ii) and $t\in C\triangle C'$, then $w(\mb y;t)$ is both strictly greater than and strictly less than $\frac{d}{p}$.  This is impossible, so the intersection is empty.

Part (iii) follows immediately from the relevant definition.
	
To prove Part (iv), let $\mb x\in \Bias(C,k,d)$ and $\mb y\in H_{l}$, with the aim of showing $\mb x + \mb y\in \Bias(C,k-l,d)$.  Then $\mb x$ satisfies $w(\mb x;t)>\frac{d}{p}+k$ for every $t\in C$ and $w(\mb x;t)<\frac{d}{p}-k$ for every $t\notin C$, while $\mb y$ has at most $l$ nonzero entries. Thus $\mb x+\mb y$ differs from $\mb x$ in at most $l$ coordinates, so that $|w(\mb x;t)-w(\mb x+\mb y;t)|\leq l$ for each $t\in \mathbb Z/p\mathbb Z$. The conditions on $w(\mb x;t)$ then imply $w(\mb x+\mb y;t)<  \frac{d}{p} - k+l$ for each $t\in C$ and $w(\mb x+\mb y;t)> \frac{d}{p} + k- l$ for each $t\notin C$.  Thus $\mb x+\mb y\in \Bias(C,k-l,d)$.
\end{proof}

\subsection{Assembling bias cells}\label{sec:Assembly}  Note that $\mathbb Z/p\mathbb Z$ acts on $\mathcal P$ by translation; call this action $\tau$. Every $C\in \mathcal P$ belongs to a $\tau$-orbit of cardinality $p$, since every $C\in \mathcal P$ satisfies $C\neq C+1$, and the cardinality of an orbit divides the order of the acting group; this is the only place where we use the primeness of $p$.  Choose a collection of sets $\mathcal P_0$ representing each $\tau$-orbit (i.e.~every $\tau$-orbit contains exactly one element of $\mathcal P_0$), so that
\[
\{\mathcal P_0, \mathcal P_0+1, \dots, \mathcal P_0+(p-1)\}
\]
is a partition of $\mathcal P$. We fix this choice of $\mathcal P_0$ for the remainder of the section.

For example, when $p=3$, we have
\[\mathcal P = \{\{0\}, \{1\}, \{2\}, \{0,1\}, \{0,2\}, \{1,2\}\},\] and we choose $\mathcal P_0 = \{\{0\},\{0,1\}\}$.  Then $\mathcal P_0+1=\{\{1\},\{1,2\}\}$ and $\mathcal P_0+2 = \{\{2\},\{0,2\}\}.$

Lemma \ref{lem:GpTile} will be proved by taking $A$ to be the following:
\begin{equation}\label{eqn:E0def}
E_0(k,d):=\bigcup_{C\in \mathcal P_0} \Bias(C,k,d).
\end{equation}
We write $E(k,d)$ for the union of \emph{all} the bias cells:
\begin{equation}\label{eqn:Edef}
  E(k,d):=\bigcup_{C\in \mathcal P} \Bias(C,k,d).
\end{equation}
We will see in Lemma \ref{lem:CountBkd} that $E(k,d)$ is the disjoint union of the translates $E_0(k,d)+n\mb 1$, $0\leq n \leq p-1$, so the following lemma will let us estimate $|E_0(k,d)|$.

\begin{lemma}\label{lem:CountE}
For fixed $k\in \mathbb N$, $\varepsilon>0$, and $d$ sufficiently large depending on $k$ and $\varepsilon$, we have
\[|E(k,d)|>(1-\varepsilon)|G_p^d|.\]
\end{lemma}

\begin{proof}
	We will prove that $E'(k,d):=G_p^d\setminus E(k,d)$ satisfies $\lim_{d\to\infty} |E'(k,d)|/|G_p^d|=0$, which is equivalent to the statement of the lemma.  Note that $E'(k,d)$ is the set of elements $(x_1,\dots, x_d)$ such that $|w(\mb x;t)-\frac{d}{p}|\leq k$ for some $t\in \mathbb Z/p\mathbb Z$.  To estimate $|E'(k,d)|$ it therefore suffices to fix $t\in \mathbb Z/p\mathbb Z$ and $m\in [\tfrac{d}{p}-k, \tfrac{d}{p}+k]$ and count the number of $\mb x$ with $w(\mb x;t)=m$.  The collection of such $\mb x$ can be enumerated by choosing $m$ coordinates of $\mb x$ to be equal to $t$, and filling in the remaining $d-m$ coordinates with any of the $p-1$ elements of $\mathbb Z/p\mathbb Z$ besides $t$.  The number of $\mb x$ such that $w(\mb x;t)=m$ is therefore $(p-1)^{d-m}\binom{d}{m}$.  Summing over the relevant values of $m$ and $t$, we find that
\begin{equation}\label{eqn:EprimeEstimate}
|E'(k,d)|\leq p(2k+1)(p-1)^{d}M_d,
\end{equation}
where $M_d=\max\{\binom{d}{m} : m\leq \tfrac{d}{p}+k \}$. Estimating the binomial coefficients in $M_d$ with Stirling's formula, we have $\lim_{d\to \infty} (p-1)^{d}M_d/p^d=0$ (remembering that $p$ and $k$ are fixed).  Inequality (\ref{eqn:EprimeEstimate}) then implies $\lim_{d\to \infty} |E'(k,d)|/|G_p^d|=0$, as desired. \end{proof}

\begin{lemma}\label{lem:CountBkd}
With $E_0(k,d)$ and $E(k,d)$ as defined in \textup{(\ref{eqn:E0def})} and \textup{(\ref{eqn:Edef})},
\begin{enumerate}
\item[(i)] for all $k, l, d\in \mathbb N$ with $l<k$, we have
\[E_0(k,d)+H_{l}\subseteq E_0(k-l,d).\]

\item[(ii)]	The translates $E_0(k,d), E_0(k,d)+\mb 1,\dots, E_0(k,d)+(p-1)\mb 1$ partition $E(k,d)$.

\item[(iii)] For fixed $k\in \mathbb N $,  $\varepsilon >0$, and sufficiently large $d$, we have \[|E_0(k,d)|> \frac{1-\varepsilon}{p}|G_p^d|.\]
\end{enumerate}

\end{lemma}

\begin{proof} Part (i) follows from the definition of $E_0(k,d)$ and Part (iv) of Lemma \ref{lem:Bias}.

Now to prove Part (ii).  To show that the sets $E_0(k,d), E_0(k,d)+\mb 1, \dots,E_0(k,d)+(p-1)\mb 1$ are mutually disjoint, fix $n\neq m \in \{0,1,\dots, p-1\}$.  We will show that $E_0(k,d)+n\mb 1$ is disjoint from $E_0(k,d)+m\mb 1$. It suffices to prove that if $C, C'\in \mathcal P_0$ (not necessarily distinct), then $\Bias(C+n,k,d)$ is disjoint from $\Bias(C'+m, k,d)$, as Part (i) of Lemma \ref{lem:Bias} implies
\[E_0(k,d)+n=\bigcup_{C\in \mathcal P_0} \Bias(C+n, k,d), \qquad E_0(k,d)+m=\bigcup_{C\in \mathcal P_0} \Bias(C+m, k,d).\] Our choice of $\mathcal P_0$ implies that if $C, C'\in \mathcal P_0$ and $m\neq n$, we have $C+n\neq C'+m$.  Part (ii)  of Lemma \ref{lem:Bias} then implies $\Bias(C+n,k,d)\cap \Bias(C'+m,k,d)=\varnothing$, as desired.

To see that the union of the translates $E_0(k,d)+m\mb 1$ is $E(k,d)$, it suffices to prove that for each $C\in \mathcal P$, there is an $m$ such that $\Bias(C,k,d)\subseteq E_0+m\mb 1$.  Our choice of $\mathcal P_0$ means that for all $C\in \mathcal P$, there exists $m\in \mathbb Z/p\mathbb Z$ such that $C-m\in \mathcal P_0$, and the definition of $E_0$ means that $\Bias(C-m,k,d)\subseteq E_0$.  We then have $\Bias(C-m,k,d)+m\mb 1\subseteq E_0+m\mb 1$, and Lemma \ref{lem:Bias} simplifies the left hand side of this containment to $\Bias(C,k,d)$.  We have therefore shown $\Bias(C,k,d)\subseteq E_0+m\mb 1$, as desired.

Finally, the estimate in Part (iii) follows from the estimate on $E(k,d)$ in Lemma \ref{lem:CountE} and the fact that the translates $E_0(k,d)+m\bm 1$, $0\leq m \leq p-1$, partition $E(k,d)$ and all have the same cardinality. \end{proof}

\begin{proof}[Proof of Lemma \ref{lem:GpTile}]To prove Lemma \ref{lem:GpTile}, we fix $k\in \mathbb N$, a prime $p$, and $\varepsilon>0$.   Use part (iii) of  Lemma \ref{lem:CountBkd} to choose $d$ sufficiently large that $|E_0(k+1,d)|>\frac{1-\varepsilon}{p}|G_p^d|$.  Let $A=E_0(k+1,d)$, and let $A_1=E_0(1,d)$.  Now Part (i) of Lemma \ref{lem:CountBkd} implies $A+H_k\subseteq A_1$, and Part (ii) of Lemma \ref{lem:CountBkd} implies that the sets $A_1+n\mb 1$, $0\leq n \leq p-1$, are mutually disjoint.  The containment $A\subseteq A_1$ follows from the containment $A+H_k\subseteq A_1$ and the fact that $\mb 0 \in H_k$.  This completes the proof of Lemma \ref{lem:GpTile}.
\end{proof}

\section{Copying sets from \texorpdfstring{$G_N^d$}{Gpd} into \texorpdfstring{$\mathbb T^d$}{Td}}\label{sec:Copying}
Fix $N, d\in \mathbb N$.  As in the previous section, $G_N^d$ is the group $(\mathbb Z/N\mathbb Z)^d$.  In this section we  present a standard way of associating subsets of $\mathbb T^d$  to subsets of  $G_N^d$.  Under this association, the containment $A+H_k\subseteq A_1$ in Lemma \ref{lem:GpTile} will yield near invariance of the associated sets under translation by elements of an \emph{approximate Hamming ball}, which we describe in Definition \ref{def:ApproxHamm}.  The near invariance mentioned here is proved in Lemma \ref{lem:HammingToApproximateHamming}.

Write $\mu$ for Haar probability measure on $\mathbb T^d$. Let $\phi:G_N^d\to \mathbb T^d$ be the homomorphism given by
\[
\phi(x_1,\dots, x_d):= (x_1/N, \dots, x_d/N).
\]
For a given $\varepsilon\geq 0$, let \[Q_{N,\varepsilon}:=\bigl[\varepsilon,\tfrac{1}{N}-\varepsilon\bigr)^d\subseteq \mathbb T^d.\] This is simply a half-open cube of side length $\frac{1}{N}-2\varepsilon$.  If $A\subseteq G_N^d$, define $A^{\ssquare}_{\varepsilon}\subseteq \mathbb T^d$ by \[A^{\ssquare}_{\varepsilon}:=\phi(A)+Q_{N,\varepsilon},\]
so that $A^{\ssquare}_{\varepsilon}$ is a disjoint union of translates of $Q_{N,\varepsilon}$.  The cubes $Q_{N,0}$ tile $\mathbb T^d$: we have $\mathbb T^d = \phi(G_N^d) + Q_{N,0}$.
The next lemma records the basic properties of this construction.

\begin{lemma}\label{lem:CopyBasic} Let $A, B\subseteq G_N^d$ and $\varepsilon\geq 0$.  Then
\begin{enumerate}
\item[(i)]  $(A\cap B)^{\ssquare}_\varepsilon = A^{\ssquare}_\varepsilon\cap B^{\ssquare}_\varepsilon$,
	
\item[(ii)]  $\mu(A^{\ssquare}_\varepsilon) = |A|(\frac{1}{N}-2\varepsilon)^d$,

\item[(iii)] $\lim_{\varepsilon\to 0^+} \mu(A^{\ssquare}_\varepsilon) = \mu(A^{\ssquare}_0) = |A|N^{-d}$,

\item[(iv)] If $\delta\leq \varepsilon$ and $\mb y\in \mathbb T^d$ has $\|\mb y\|< \delta$, then $ A^{\ssquare}_\varepsilon + \mb y \subseteq  A^{\ssquare}_{\varepsilon-\delta}$.
\end{enumerate}
\end{lemma}

\begin{proof} Part (i) follows immediately from the definitions.  We get Part (ii) by observing that $A^{\ssquare}_\varepsilon$ is a disjoint union of $|A|$ cubes in $\mathbb T^d$ having side length $\frac{1}{N}-2\varepsilon$.  Part (iii) follows immediately from Part (ii). Part (iv) follows from the observation that if $\|\mb y\|< \delta$, then $Q_{N,\varepsilon}+\mb y\subseteq  Q_{N,\varepsilon-\delta}$.  \end{proof}

The important consequence of Part (i) in Lemma \ref{lem:CopyBasic} is that when $A_1, \dots, A_j  \subseteq  G_N^d$ are mutually disjoint, the corresponding sets $(A_1)^{\ssquare}_0,\dots, (A_j)^{\ssquare}_0 \subseteq \mathbb T^d$ are mutually disjoint.

 Recall from \S\ref{sec:BHsubsec}  that for $\mb x\in \mathbb T^d$ and $\varepsilon>0$, we defined   $w_{\varepsilon}(\mb x):=|\{j:\|x_j\|\geq \varepsilon\}|$.

\begin{definition}\label{def:ApproxHamm}
  For $k< d\in \mathbb N$ and $\varepsilon>0$, we define the \emph{approximate Hamming ball} of radius $(k,\varepsilon)$ around $0_{\mathbb T^d}$ as
  \[
  \Hamm(k,\varepsilon):=\{\mb x\in \mathbb T^d: w_{\varepsilon}(\mb x)\leq k\}.
  \]
So $\Hamm(k,\varepsilon)$ is the set of $\mb x=(x_1,\dots,x_d)\in \mathbb T^d$ where at most $k$ coordinates differ from $0$ by at least $\varepsilon$.
\end{definition}

The following lemma is crucial in deriving the containment $R^nE_0\subseteq E_1$ (for $n\in BH$) in Proposition \ref{prop:RohlinTd} from the containment $A+H_k\subseteq A_1$ in Lemma \ref{lem:GpTile}.  For a set $B\subseteq \mathbb T^d$ we use $\bar{B}$ to denote its topological closure.

\begin{lemma}\label{lem:HammingToApproximateHamming}
	Let $\varepsilon\geq \eta >0$, $k< d\in \mathbb N$, and   $A\subseteq G_N^d$.  Let $U:=\Hamm(k,\eta)\subseteq \mathbb T^d$ and $H:=H_k\subseteq G_N^d$, as in \S\textup{\ref{sec:HammField}}.  Then
	\begin{enumerate}
		\item[(i)] $A^{\ssquare}_\varepsilon+U\subseteq  (A+H)^{\ssquare}_{0}$.
	
	\item[(ii)]  If $\varepsilon>\eta$, then $\overline{A^{\ssquare}_\varepsilon+U}\subseteq  (A+H)^{\ssquare}_{0}.$
	\end{enumerate}
\end{lemma}

\begin{proof}
To prove Part (i), note that the left hand side therein is $\phi(A)+Q_{N,\varepsilon}+U$, and the right hand side simplifies as $\phi(A + H)+Q_{N,0}=\phi(A)+\phi(H)+Q_{N,0}$.  It therefore suffices to prove that
\begin{equation}\label{eqn:QUpHQ}
  Q_{N,\varepsilon}+U\subseteq  \phi(H)+Q_{N,0}.
\end{equation}
To prove this containment, let $\mb u\in U$ with the aim of showing $Q_{N,\varepsilon}+\mb u\subseteq  \phi(H)+Q_{N,0}$.  This $\mb u$ can be written as $\mb y + \mb z$, where $\|\mb y\|<\eta$ and $\mb z=(z_1,\dots,z_d)$ has at most $k$ nonzero coordinates.  Part (iv) of Lemma \ref{lem:CopyBasic} implies $Q_{N,\varepsilon}+\mb y + \mb z  \subseteq  Q_{N,0}+\mb z$,  so we must show that
\begin{equation}\label{eqn:QwPhiQ}
Q_{N,0}+\mb z \subseteq  \phi(H)+Q_{N,0}.
\end{equation}
The left hand side above is the set of $\mb x$ in $\mathbb T^d$ where at most $k$ coordinates of $\mb x$ lie outside $\bigl[0,\frac{1}{N}\bigr)$.  Fixing such an $\mb x$ as $(x_1,\dots,x_d)$, we will show that $\mb x\in \phi(H)+Q_{N,0}$.  For each $j$, choose $h_j\in \{0,\dots,N-1\}$ so that $x_j\in \bigl[\frac{h_j}{N}, \frac{h_j+1}{N}\bigr)$.  Then $h_j\neq 0$ for at most $k$ indices $j$, since at most $k$ coordinates of $\mb x$ lie outside $\bigl[0,\frac{1}{N}\bigr)$.  Setting $\mb h= (h_1 \operatorname{mod} N,\dots, h_d \operatorname{mod} N)$, we have $\mb h\in H$, and $\mb x\in \phi(\mb h)+Q_{N,0}$.  This proves the containment (\ref{eqn:QwPhiQ}), and therefore establishes (\ref{eqn:QUpHQ}), concluding the proof of Part (i).

To prove Part (ii), assume $\varepsilon>\eta>0$, and choose $\varepsilon'$ and $\eta'$ so that $\varepsilon>\varepsilon'>\eta'>\eta$.  Let $U':=\Hamm(k,\eta')$.  Our choice of $\varepsilon'$ and $\eta'$ means that $\overline{A^{\ssquare}_{\varepsilon}}\subseteq A^{\ssquare}_{\varepsilon'}$ and $\overline{U}\subseteq U'$.  We therefore have $\overline{A^{\ssquare}_{\varepsilon} + U} = \overline{A^{\ssquare}_{\varepsilon}}+\overline{U} \subseteq A^{\ssquare}_{\varepsilon'}+U'\subseteq (A+H)^{\ssquare}_{0}$, where the last containment is an instance of Part (i). \end{proof}

\section{Rohlin towers for torus rotations; proof of Proposition \ref{prop:RohlinTd}}\label{sec:RohlinTowers}

The following lemma is a restatement of Proposition \ref{prop:RohlinTd}.  It is proved by associating the sets provided by Lemma \ref{lem:GpTile} to subsets of $\mathbb T^d$ using the machinery of \S\ref{sec:Copying}.

\begin{lemma}\label{lem:TileT}
For all $k\in \mathbb N$, every prime $p$, and all $\varepsilon>0$, there exist $d\in \mathbb N$, $\eta>0$, sets $E, E'\subseteq \mathbb T^d$, a generator $\bm\alpha\in \mathbb T^d$, and an approximate Hamming ball $U:=\Hamm(k,\eta)\subseteq \mathbb T^d$ such that
\begin{enumerate}
  \item[(i)] the translates  $E', E'+\bm\alpha, \dots, E'+(p-1)\bm\alpha$,
 are mutually disjoint,
 \item[(ii)] $\mu(E)> \frac{1-\varepsilon}{p}$, and

  \item[(iii)]  $E+U\subseteq E'$.
\end{enumerate}
Consequently, the Bohr-Hamming ball $BH:=BH(\bm\alpha;k,\eta)$ satisfies $E+BH\bm\alpha\subseteq E'$, and thus $E\subseteq E'$.
\end{lemma}
Proposition \ref{prop:RohlinTd} follows from Lemma \ref{lem:TileT}, as Part (i) here asserts that $\{R^nE':0\leq n \leq p-1\}$ is a Rohlin tower for the torus rotation on $\mathbb T^d$ by $\bm\alpha$, and the containment $E\subseteq E'$ then implies $\{R^nE : 0\leq n \leq p-1\}$ is a Rohlin tower as well.  The containment $E+BH\bm\alpha$ here is the part of Proposition \ref{prop:RohlinTd} asserting $R^nE\subseteq E'$ for all $n\in BH$.

\begin{proof}  Fix $k\in \mathbb N$, a prime $p$, and $\varepsilon>0$.  By Lemma \ref{lem:GpTile}, choose $d$ sufficiently large and $A$, $A_1\subseteq (\mathbb Z/p\mathbb Z)^d$ such that  the translates
\begin{equation}\label{eqn:AHDisjoint}
A_1, A_1+\mb 1,\dots, A_1+(p-1)\mb 1 \quad \text{are mutually disjoint},
\end{equation} $|A|>\frac{1-\varepsilon/2}{p}p^d$, $A\subseteq A_1$, and $A+H_k\subseteq A_1$. We fix these choices of $A$ and $A_1\subseteq (\mathbb Z/p\mathbb Z)^d$ and use them to select $\bm\alpha\in \mathbb T^d$, $E$, and $E'\subseteq \mathbb T^d$.  We use the definitions of $\phi$ and $(\cdot)^{\square}_{\varepsilon}$ established in \S\ref{sec:Copying},  so that $\phi:(\mathbb Z/p\mathbb Z)^d\to \mathbb T^d$ and $\phi(\mb 1)=(1/p,\dots,1/p)\in \mathbb T^d$.

The disjointness in (\ref{eqn:AHDisjoint}) and Part (i) of Lemma \ref{lem:CopyBasic} imply that the sets $(A_1+n\mb 1)^{\ssquare}_{0}$, $n\in \{0,\dots,p-1\}$, are mutually disjoint.  By our definition of $(\cdot)^{\ssquare}_{0}$ and $\phi$, this means that the translates \begin{equation}\label{eqn:DisjointA1square}
(A_1)^{\ssquare}_{0},\, (A_1)^{\ssquare}_{0} + \phi(\mb 1),\, \dots\, ,\, (A_1)^{\ssquare}_{0} + (p-1)\phi(\mb 1) \quad \text{are mutually disjoint.}
\end{equation}

Now we specify the sets $E$, $U$, and $E'$.  Part (iii) of Lemma \ref{lem:CopyBasic} provides an $\eta>0$ satisfying $\mu(A^{\ssquare}_{2\eta})> (1-\frac{\varepsilon}{2})|A|p^{-d}$; our choices of $\eta$ and $A$ then guarantee that $\mu(A^{\ssquare}_{2\eta})> \frac{1-\varepsilon}{p}$.  Let
\[
E:=A^{\ssquare}_{2\eta}, \qquad U:=\Hamm(k,\eta), \qquad   E':=\overline{A^{\ssquare}_{2\eta}+U}.
\]
 Parts (ii) and (iii) of the present lemma are evidently satisfied by this  $E$.  To find $\bm\alpha$ satisfying Part (i), we first observe that Lemma \ref{lem:HammingToApproximateHamming} implies $E'\subseteq (A+H_k)^{\ssquare}_{0}$. Our choice of $A$ and $A_1$ then implies $E'\subseteq (A_1)^{\ssquare}_{0}.$  This containment and (\ref{eqn:DisjointA1square}) imply that the translates
 \begin{equation}\label{eqn:E1phi1}
 	E', E'+\phi(\mb 1),\dots, E'+(p-1)\phi(\mb 1)
 \end{equation} are mutually disjoint.  They are all compact, as well, so for every $\bm\alpha$ sufficiently close to $\phi(\mb 1)$, the translates
\begin{equation}\label{eqn:phi1y}
E', E'+\bm\alpha, \dots, E'+(p-1)\bm\alpha \quad \text{are mutually disjoint.}
\end{equation}
  In particular, we can choose such an $\bm\alpha$ to be a generator, and with this $\bm\alpha$ the disjointness in (\ref{eqn:phi1y}) implies the disjointness asserted in Part (i) of the lemma.

Finally, the containment $E+BH\bm\alpha\subseteq E'$ follows from the containment $E+U\subseteq E'$ and the fact that $n\in BH(\bm\alpha;k,\eta)$ if and only if $n\bm\alpha \in U$.
\end{proof}

\section{Proof of Lemma \ref{lem:BHiskBohrDense}}\label{sec:Independence}

As in \S\ref{sec:BohrNhoods}, for a given  $\alpha\in \mathbb T$ we write $\tilde{\alpha}$ for its representative in $[0,1)\subseteq \mathbb R$.  Given $E\subseteq \mathbb R$, we write $\operatorname{span}E$ for the set of rational linear combinations of elements of $E$ (i.e. the $\mathbb Q$-linear span). We will abbreviate sets of indexed elements $\{x_1,\dots,x_d\}$ as $\{x_i\}$; for example, $\operatorname{span}\{x_1,\dots,x_d\}$ may be written as $\operatorname{span}\{x_i\}$. We also suppress the index of summation in sums, where it will cause no confusion.

We need the following standard facts from harmonic analysis, presented in references such as \cite{FollandHABook,KaznelsonHA,MorrisDualityBook,RudinGroupsBoook}. An \emph{additive character} of $\mathbb T^d$ is a continuous homomorphism from $\mathbb T^d$ to $\mathbb T$.

\begin{enumerate}
  \item[$\bullet$] Every additive character of $\mathbb T^d$ has the form $(x_1,\dots,x_d)\mapsto \sum n_ix_i$ for some fixed $(n_1,\dots,n_d)\in \mathbb Z^d$.
  \item[$\bullet$] If $K$ and $L$ are closed subgroups of $\mathbb T^d$ and $K$ is properly contained in $L$, then there is an additive character of $\mathbb
T^d$ which vanishes on $K$ and not on $L$.
\item[$\bullet$] (Kronecker's characterization) $\bm\alpha = (\alpha_1,\dots,\alpha_d)\in \mathbb T^d$ is a generator if and only if $\{\tilde{\alpha}_1,\dots,\tilde{\alpha}_d,1\}$ is linearly independent over $\mathbb Q$.
\end{enumerate}

The next lemma is a variant of Kronecker's characterization, and is proved in essentially the same way.  Given an element $g$ of a group, we use $\langle g\rangle$ to denote the subgroup generated by $g$.  We write $\overline{A}$ for the topological closure of a subset $A$ of a topological space.

\begin{lemma}\label{lem:RelativeKronecker}
  Let $r,k\in \mathbb N$, and suppose $\bm\alpha=(\alpha_1,\dots,\alpha_r)\in \mathbb T^r$ and $\bm\beta=(\beta_1,\dots,\beta_k)\in \mathbb T^k$  are such that   $\operatorname{span}\{\tilde{\alpha}_i\}\cap \operatorname{span}(\{\tilde{\beta}_i\}\cup \{1\})=\{0\}$. Then
  \begin{enumerate}
  \item[(i)] in $\mathbb T^{r}\times \mathbb T^{k}$ we have
  $\overline{\langle (\bm\alpha,\bm\beta) \rangle} = \overline{\langle \bm\alpha\rangle}\times \overline{\langle \bm\beta\rangle}$.
  \item[(ii)] If $U\subseteq \mathbb T^k$, $V\subseteq \mathbb T^r$ are open and $B(\bm\alpha;U)$, $B(\bm\beta;V)$ are nonempty Bohr neighborhoods, then $B((\bm\alpha,\bm\beta);U\times V)$ is a nonempty Bohr neighborhood of rank $r+k$.
  \end{enumerate}
  \end{lemma}

\begin{proof}
We prove the contrapositive of Part (i).  Assuming $\overline{\langle (\bm\alpha,\bm\beta) \rangle}\neq \overline{\langle \bm\alpha\rangle}\times \overline{\langle \bm\beta\rangle}$, the former must be properly contained in the latter, since $\langle \bm\alpha,\bm\beta\rangle \subseteq \langle \bm\alpha\rangle\times \langle \bm\beta\rangle$.  This means there is an additive character of $\mathbb T^r\times \mathbb T^k$ vanishing on $\overline{\langle (\bm\alpha,\bm\beta) \rangle}$ and not on $\overline{\langle \bm\alpha\rangle}\times \overline{\langle \bm\beta\rangle}$.  In other words, there exists $(n_1,\dots, n_r,m_1,\dots,m_k)\in \mathbb Z^{r+k}$ such that $\sum_{i,j} n_i\alpha_i + m_j\beta_j=0 \in \mathbb T$, and at least one of $\sum n_i\alpha_i$, $\sum m_j\beta_j$ is nonzero.  Together this implies both $\sum n_i\alpha_i$ and $\sum m_j\beta_j$ are nonzero.  Thus there is an integer $N$ such that $\sum n_i\tilde{\alpha}_i = N - \sum m_j\tilde{\beta}_j$, meaning the respective spans of $\{\tilde{\alpha}_1,\dots,\tilde{\alpha}_r\}$ and $\{\tilde{\beta}_1,\dots,\tilde{\beta}_d,1\}$ intersect nontrivially.

To prove Part (ii), assume $B(\bm\alpha;U)$ and $B(\bm\beta;V)$ are both nonempty.  Note that $B((\bm\alpha,\bm\beta);U\times V)\neq \varnothing$  if and only if the intersection $Y:=(U \times V) \cap \overline{\langle (\bm\alpha,\bm\beta)\rangle}$ is nonempty (since $U\times V$ is open).  The same observation implies $\overline{\langle \bm\alpha\rangle}\cap U$ and $\overline{\langle \bm\beta\rangle}\cap V$ are both nonempty, by the hypothesis of Part (ii). The set $Y$ is therefore nonempty, and this yields the conclusion.
\end{proof}

\begin{proof}[Proof of Lemma \ref{lem:BHiskBohrDense}]
The first assertion of the lemma states that if $k<d\in \mathbb N$, $\varepsilon>0$, and $BH$ is a proper Bohr-Hamming ball with rank $d$ and radius $(k,\varepsilon)$ and $B$ is a nonempty Bohr neighborhood with rank $k$, then $BH\cap B$ contains a Bohr neighborhood with rank $d$.  To prove this, fix such $BH$ and $B$, write $\bm\alpha = (\alpha_1,\dots,\alpha_d)$ for the frequency determining $BH$, and write $B$ as $\{n\in \mathbb Z: n\bm\beta \in V\}$, where $V\subseteq \mathbb T^k$ is open and $\bm\beta=(\beta_1,\dots,\beta_k)\in \mathbb T^k$.

Let $W$, $Y$, and $Z$ denote $\operatorname{span}(\{\tilde{\alpha}_i\}\cup\{1\})$, $\operatorname{span}\{\tilde{\beta}_i\}$, and $\operatorname{span}(\{\tilde{\alpha}_i\}\cup\{\tilde{\beta}_i\}\cup \{1\})$, respectively.  Our assumption that $BH$ is proper means that $\bm\alpha$ is a generator, so $W$ has dimension $d+1$, and the dimension of $Z$ is at least $d+1$.

  We select $d-k$ elements $\tilde{\alpha}_{i_1},\dots,\tilde{\alpha}_{i_{d-k}}$ of the $\tilde{\alpha}_i$ so that $\operatorname{span}\{\tilde{\alpha}_{i_j}\}\cap \operatorname{span}(\{\tilde{\beta}_i\}\cup\{1\})=\{0\}$. To make this selection, first choose a basis $\mathcal B$ for $Y$ from among the $\tilde{\beta}_i$, then extend $\mathcal B$ to a basis $\mathcal B'$ for $Z$ by adjoining elements of $\{\tilde{\alpha}_i\}\cup\{1\}$.  This basis must contain at least $d+1$ elements, since $Z$ has dimension $\geq d+1$, so $\mathcal B'$ contains at least $d-k$ elements $\tilde{\alpha}_{i_1},\dots,\tilde{\alpha}_{i_{d-k}}$ of the $\tilde{\alpha}_i$.  Since $\mathcal B'$ is linearly independent and the $\tilde{\alpha}_{i_j}$ are disjoint from $\mathcal B \cup\{1\}$, we get  that $\operatorname{span}\{\tilde{\alpha}_{i_j}\}\cap \operatorname{span}(\{\tilde{\beta}_i\}\cup \{1\})$ is trivial.

Let $\bm\alpha'=(\alpha_{i_1},\dots,\alpha_{i_{d-k}})\in \mathbb T^{d-k}$, and let $C$ be the basic Bohr neighborhood $B(\bm\alpha';U)$, where $U$ is the $\varepsilon$-ball around $0$ in $\mathbb T^{d-k}$. Comparing their respective definitions, we see that $BH$ contains $C$; furthermore $C$ is nonempty, as $0\in C$.  Thus, $BH\cap B$ contains $C\cap B=B((\bm\alpha',\bm\beta);U\times V)$.  The latter set is a Bohr neighborhood with rank $d$ which, by Lemma \ref{lem:RelativeKronecker}, is nonempty.

  The second assertion of Lemma \ref{lem:BHiskBohrDense} follows from the first: if $S$ is $d$-Bohr dense and $B$ is a Bohr neighborhood with rank $k$, then $S\cap (BH\cap B)\neq \varnothing$, by virtue of the fact that $BH\cap B$ contains a nonempty Bohr neighborhood of rank $d$.  Thus $(S\cap BH)\cap B\neq \varnothing$ for every Bohr neighborhood $B$ with rank $k$, so $S\cap BH$ is $k$-Bohr dense.  \end{proof}

\section{Proof of Lemmas \ref{lem:nonrecurrenceForms} and \ref{lem:DensityCompactness}}\label{sec:Appendix}

 Lemma \ref{lem:nonrecurrenceForms} says that for $\delta>0$ and $S\subseteq \mathbb Z$, the following conditions are equivalent.
  \begin{enumerate}
    \item[(i)] There is a measure preserving system $(X,\mu,T)$ and $D\subseteq X$ with $\mu(D)>\delta$ such that $\mu(D\cap T^sD)=\varnothing$ for all $s\in S$.

    \item[(ii)] There exists $A\subseteq \mathbb Z$ with $d^*(A)>\delta$ such that $(A-A)\cap S=\varnothing$.

    \item[(iii)] There is a $\delta'>\delta$ such that for all $n\in \mathbb N$ there exist $A_n\subseteq \{0,\dots,n-1\}$  with $|A_n|\geq \delta'n$ and $(A_n-A_n)\cap S=\varnothing$.
  \end{enumerate}

\begin{proof}   To prove (i)$\implies$(ii), let $\delta>0$, $S\subseteq \mathbb Z$, $(X,\mu,T)$, and $D$ satisfy (i).  We will find  $A\subseteq \mathbb Z$ with $d^*(A)>\delta$ and $A\cap (A+S)=\varnothing$.
	
Write $1_D$ for the characteristic function of $D$.  By the pointwise ergodic theorem, the limit
	\begin{equation}\label{eqn:PWErgodic}
	F(x):=\lim_{N\to \infty} \frac{1}{N}\sum_{n=0}^{N-1} 1_D(T^nx)
	\end{equation}
	exists for $\mu$-almost every $x$.   The dominated convergence theorem implies  $\int F\, d\mu=\int 1_D\, d\mu$, so the limit on the right hand side of (\ref{eqn:PWErgodic}) is greater than $\delta$ for some $x\in X$.  Fixing such $x$ and setting
\[A:=\{n: T^nx\in D\},\] we have $\lim_{N\to \infty} \frac{|A\cap \{0,\dots,N-1\}|}{N}=F(x)$, so $d^*(A)>\delta$.  	
To prove that $(A-A)\cap S=\varnothing$, note that $m\in A\cap (A+s)$ if and only if $T^mx\in D$ and $T^{m-s}x\in D$, so $T^m x\in D\cap T^sD$.  Our hypothesis that $D\cap T^sD=\varnothing$ for all $s\in S$ then implies $A\cap (A+s)=\varnothing$ for all $s\in S$, meaning $(A-A)\cap S=\varnothing$.

To prove (ii)$\implies$(iii), assume $A\subseteq \mathbb Z$ has $d^*(A)=\delta''>\delta'>\delta$.  Then there are intervals $I_k$ with $|I_k|\to \infty$ such that $|A\cap I_k|\geq \delta'|I_k|$ for all sufficiently large $k$.

 Fix $n\in \mathbb N$, and write $I_k$ as a union of mutually disjoint intervals $J_{k,1},\dots,J_{k,r}$ of length $n$ together with one (possibly empty) interval $J_{k,0}$ of length at most $n$.  Observe that  $|I_k|/r\geq n$ under this arrangement.  Then
 \[|A\cap I_k|=\sum_{i=0}^r |A\cap J_{k,i}|\geq \delta'|I_k|,\] so $\sum_{i=1}^r |A\cap J_{k,i}|\geq \delta'|I_k| - n$.  This implies that for some $i\leq r$, $|A\cap J_{k,i}| \geq \frac{1}{r}(\delta'|I_k|-n)\geq \delta'n -\frac{n}{r}$.  Since $|A\cap J_{k,i}|$ is integer valued and $n/r\to 0$ as $k\to\infty$, we get that $|A\cap J_{k,i}|\geq \delta'n$ for sufficiently large $k$.  We have thus found an interval $J=J_{k,i}$ of length $n$ with $|A\cap J|\geq \delta'n$. We let $A_n= (A\cap J)-\min(J)$, so that $A\subseteq \{0,\dots,n-1\}$, and $A_n-A_n \subseteq A - A$. The assumption that $(A-A)\cap S=\varnothing$ allows us to conclude (iii).

Finally we prove (iii) $\implies$ (i).  Let $\delta'>\delta>0$ and let $A_n\subseteq \{0,\dots,n-1\}$ have $|A_n|\geq \delta' n$. Consider the topological space $X:=\{0,1\}^\mathbb Z$ with the product topology, so that $X$ is a compact metrizable space. Let $T$ be the continuous transformation $T:X\to X$ given by $(Tx)(m)=x(m+1)$ (commonly known as the left shift).

  Let $E$ be the clopen set $\{x\in X: x(0)=1\}$. For each $n$, let $y_n=1_{A_n}\in X$, meaning $y_n$ is the characteristic function of $A_n$, viewed as an element of $X$.  Note that
\begin{enumerate}
  \item[(a)] $T^{m}y_n\in E$ if and only if $m\in A_n$, and

  \item[(b)] $T^{m}y_n\in E\cap T^k E$ if and only if $m\in A_n$ and $m-k\in A_n$, meaning $k\in A_n-A_n$.  In particular, if $k\notin A_n-A_n$, then $T^my_n\notin E\cap T^{k}E$ for all $m\in \mathbb Z$.
\end{enumerate}
For $x\in X$, let $\delta_x$ denote the Dirac probability measure concentrated at $x$.

Let $\mu_n = \frac{1}{n}\sum_{k=0}^{n-1} \delta_{T^{k}y_n}$, so that each $\mu_n$ is a Borel probability measure on $X$.  Note that $\mu_n(E):=\frac{1}{n}|\{k\in \mathbb N:T^k y_n\in E\}|=|A_n|/n$, by (a) above, so $\mu_n(E)\geq \delta'$ for each $n$.

  Let $\mu$ be a weak$^*$ limit of the $\mu_n$. It is easy to verify that $\mu$ is $T$-invariant, by proving $\int f\,d\mu=\int f\circ T\, d\mu$ for every continuous $f$: write $\int f\, d\mu - \int f\circ T\, d\mu$ as a limit of a subsequence of
  \[I_n:= \frac{1}{n}\sum_{k=0}^{n-1}f(T^k y_n)-f(T^{k+1}y_n),\] which converges to $0$ by cancellation and boundedness of $f$.

  We claim that $(X,\mu,T)$ is a measure preserving system satisfying (i) in the statement of the lemma.   To see this, first note that $\mu(E)\geq \delta'$, since $\mu_n(E)>\delta'$ for each $n$.  Fixing $s\in S$ (which is disjoint from $A_n-A_n$ by hypothesis), we have $\mu_n(E\cap T^sE)=0$ for each $n$, by observation (b).  Thus $\mu(E\cap T^sE)=0$ for all $s\in S$. Now let $D=E\setminus \bigcup_{s\in S} T^s E$, so that $\mu(D)=\mu(E)\geq \delta$.  Then $D\cap T^sD\subseteq D\cap T^sE$ since $D\subseteq E$.  Also $D$ is disjoint from $T^s E$, so we have $D\cap T^sD=\varnothing$ for every $s\in S$.  \end{proof}

\noindent Recall Lemma \ref{lem:DensityCompactness}: if $0\leq \delta<\delta'$ and every finite subset of $S\subseteq \mathbb Z$ is $\delta'$-nonrecurrent, then $S$ is $\delta$-nonrecurrent.

\begin{proof}[Proof of Lemma \ref{lem:DensityCompactness}]
Suppose $S\subseteq \mathbb Z$, $0\leq \delta<\delta'$, and that every finite subset of $S$ is $\delta'$-nonrecurrent.  Write $S$ as an increasing union $\bigcup_{k\in \mathbb N} S_k$ of finite sets $S_k$.  For each $k$, choose a set $C_k\subseteq \mathbb Z$ with $d^*(C_k)>\delta'$ such that $(C_k-C_k)\cap S_k=\varnothing$.

We will find a sequence of sets $A_n\subseteq \{0,\dots, n-1\}$ such that $|A_n|\geq \delta'n$ and for each $n$, infinitely many of the $C_k$ contain a translate $A_n+t_{k,n}$ of $A_n$.  Under these conditions, we have $A_n-A_n\subseteq C_k-C_k$ for such $k$, so that $(A_n-A_n)\cap S_k=\varnothing$ for infinitely many $k$.  Since the $S_k$ are increasing and exhaust $S$, this implies $(A_n-A_n)\cap S=\varnothing$, whereby part (iii) of Lemma \ref{lem:nonrecurrenceForms} implies $S$ is $\delta$-nonrecurrent.

We find the sets $A_n$ by fixing $n$ and choosing, for each $k$, an interval $I_k=[t_k,t_k+n-1]$ with length $n$ such that $|I_k\cap C_k|\geq \delta'n$, just as in the proof of (ii)$\implies$(iii) in Lemma \ref{lem:nonrecurrenceForms}.  Letting $C_k'=(I_k\cap C_k)-t_k$, we see that $C_k'\subseteq \{0,\dots,n-1\}$.  There are only finitely many subsets of $\{0,\dots, n-1\}$, so there is an infinite collection of indices $k$ such that the $C_k'$ are mutually identical for these $k$.  We let $A_n$ be one of these $C_k'$, so that $A_n+t_k\subseteq C_k$ for infinitely many $k$, as desired. \end{proof}

\section{Remarks and a problem}

F{\o}lner \cite{Folner} proved that if $A\subseteq \mathbb Z$ has $d^*(A)>0$, then $A-A$ contains a set $B\setminus Z$, where $B$ is a Bohr neighborhood of $0$ and $d^*(Z)=0$.  Kriz \cite{Kriz} constructed the first example  of a set $A\subseteq \mathbb Z$ having $d^*(A)>0$ such that $A-A$ does not contain a Bohr neighborhood of $0$.  Theorem \ref{thm:Main} shows that F{\o}lner's theorem cannot be improved to say that $A-A$ contains a Bohr neighborhood, even with the modification that the Bohr neighborhood may be around some nonzero $n$.

Our method is very similar to Kriz's, and to Ruzsa's  simplified version of Kriz's method presented in \cite{McCutcheonAlexandria,McCutcheonBook}: the Bohr-Hamming balls we consider are closely analogous to the embeddings of Kneser graphs used in \cite{Kriz}, and our Proposition \ref{prop:RohlinTd}  is an extreme modification of Lemma  3.2 in \cite{Kriz}.  Katznelson in \cite{Katznelson} showed that translates of Bohr-Hamming balls (absent the nomenclature) are $k$-Bohr recurrent but not $(k+1)$-Bohr recurrent.

Theorem \ref{thm:Strong} suggests the following problem.
\begin{problem}\label{prob:RelativeBohr}
  Prove that if $S\subseteq \mathbb Z$ is Bohr recurrent, then there is a set $S'\subseteq S$ such that $S'$ is Bohr recurrent and is not a set of measurable recurrence.
\end{problem}




\section*{Acknowledgments}

 The author owes thanks to Juan B\`es, Anh Le, Quentin Menet, Wenbo Sun, and especially Yunied Puig for their interest and encouragement regarding Problem \ref{q:Main}.  We thank Nishant Chandgotia and Anh Le for corrections, and we thank Benjy Weiss for pointing out the use of Bohr-Hamming balls in \cite{Katznelson}.  An anonymous referee suggested several improvements to exposition, which we incorporated gratefully.


\bibliographystyle{amsplain}

\begin{thebibliography}{10}

\bibitem{BadeaGrivauxMatheronRigidityKazhdan}
Catalin {Badea}, Sophie {Grivaux}, and Etienne {Matheron}, \emph{{Rigidity
  sequences, Kazhdan sets and group topologies on the integers}}, J. Anal.
  Math. \textbf{143} (2021), no.~2, 313--347 \MR{4299163}.

\bibitem{BHK}
Vitaly Bergelson, Bernard Host, and Bryna Kra, \emph{Multiple recurrence and
  nilsequences}, Invent. Math. \textbf{160} (2005), no.~2, 261--303, With an
  appendix by Imre Ruzsa. \MR{2138068}

\bibitem{BergelsonRuzsa}
Vitaly Bergelson and Imre~Z. Ruzsa, \emph{Sumsets in difference sets}, Israel
  J. Math. \textbf{174} (2009), 1--18. \MR{2581205 (2011b:11014)}

\bibitem{FollandHABook}
Gerald~B. Folland, \emph{A course in abstract harmonic analysis}, second ed.,
  Textbooks in Mathematics, CRC Press, Boca Raton, FL, 2016. \MR{3444405}

\bibitem{Folner}
Erling F{\o}lner, \emph{Note on a generalization of a theorem of
  {B}ogolio\`uboff}, Math. Scand. \textbf{2} (1954), 224--226. \MR{69188}

\bibitem{ForrestThesis}
Alan~Hunter Forrest, \emph{Recurrence in dynamical systems: {A} combinatorial
  approach}, ProQuest LLC, Ann Arbor, MI, 1990, Thesis (Ph.D.)--The Ohio State
  University. \MR{2685439}

\bibitem{Fbook}
H.~Furstenberg, \emph{Recurrence in ergodic theory and combinatorial number
  theory}, Princeton University Press, Princeton, N.J., 1981, M. B. Porter
  Lectures. \MR{603625}

\bibitem{RuzsaBook}
Alfred Geroldinger and Imre~Z. Ruzsa, \emph{Combinatorial number theory and
  additive group theory}, Advanced Courses in Mathematics. CRM Barcelona,
  Birkh\"auser Verlag, Basel, 2009, Courses and seminars from the DocCourse in
  Combinatorics and Geometry held in Barcelona, 2008. \MR{2547479
  (2010f:11005)}

\bibitem{GriesmerIsr}
John~T. Griesmer, \emph{Sumsets of dense sets and sparse sets}, Israel J. Math.
  \textbf{190} (2012), 229--252. \MR{2956240}

\bibitem{HegyvariRuzsa}
Norbert Hegyv{\'a}ri and Imre~Z. Ruzsa, \emph{Additive structure of difference
  sets and a theorem of {F}\o lner}, Australas. J. Combin. \textbf{64} (2016),
  437--443. \MR{3457812}

\bibitem{Katznelson}
Y.~Katznelson, \emph{Chromatic numbers of {C}ayley graphs on {$\mathbb Z$} and
  recurrence}, Combinatorica \textbf{21} (2001), no.~2, 211--219, Paul
  Erd\H{o}s and his mathematics (Budapest, 1999). \MR{1832446}

\bibitem{KaznelsonHA}
\bysame, \emph{An introduction to harmonic analysis}, third ed., Cambridge
  Mathematical Library, Cambridge University Press, Cambridge, 2004.
  \MR{2039503}

\bibitem{Kriz}
Igor Kriz, \emph{Large independent sets in shift-invariant graphs: solution of
  \text{Bergelson's} problem}, Graphs Combin. \textbf{3} (1987), no.~2,
  145--158. \MR{932131}

\bibitem{McCutcheonAlexandria}
Randall McCutcheon, \emph{Three results in recurrence}, Ergodic theory and its
  connections with harmonic analysis ({A}lexandria, 1993), London Math. Soc.
  Lecture Note Ser., vol. 205, Cambridge Univ. Press, Cambridge, 1995,
  pp.~349--358. \MR{1325710}

\bibitem{McCutcheonBook}
\bysame, \emph{Elemental methods in ergodic {R}amsey theory}, Lecture Notes in
  Mathematics, vol. 1722, Springer-Verlag, Berlin, 1999. \MR{1738544}

\bibitem{MorrisDualityBook}
Sidney~A. Morris, \emph{Pontryagin duality and the structure of locally compact
  abelian groups}, Cambridge University Press, Cambridge-New York-Melbourne,
  1977, London Mathematical Society Lecture Note Series, No. 29. \MR{0442141}

\bibitem{RudinGroupsBoook}
Walter Rudin, \emph{Fourier analysis on groups}, Interscience Tracts in Pure
  and Applied Mathematics, No. 12, Interscience Publishers (a division of John
  Wiley and Sons), New York-London, 1962. \MR{0152834}

\bibitem{RuzsaIntersectivity84}
I.~Z. Ruzsa, \emph{On measures on intersectivity}, Acta Math. Hungar.
  \textbf{43} (1984), no.~3-4, 335--340. \MR{733865}

\end{thebibliography}

\providecommand{\bysame}{\leavevmode\hbox to3em{\hrulefill}\thinspace}
\providecommand{\MR}{\relax\ifhmode\unskip\space\fi MR }
\providecommand{\MRhref}[2]{%
  \href{http://www.ams.org/mathscinet-getitem?mr=#1}{#2}
}
\providecommand{\href}[2]{#2}

%
%

\begin{dajauthors}
\begin{authorinfo}[jtg]
  John T. Griesmer\\
  Department of Applied Mathematics and Statistics\\
  Colorado School of Mines\\
  Golden, Colorado USA\\
  griesmer\imageat{}mines\imagedot{}edu \\
\end{authorinfo}
\end{dajauthors}

\end{document}